\documentclass[11pt]{article}
\usepackage[utf8]{inputenc}
\usepackage{lmodern}
\usepackage{subfiles}
\usepackage{enumitem}
\setenumerate{topsep=6pt,ref={\normalfont(\roman*)},label={\normalfont(\roman*)}, itemsep=0em} 
        
\usepackage{amsfonts}
\usepackage{amsthm}
\usepackage{amsmath}
\usepackage{amssymb}
\usepackage{amscd}
\usepackage{mathrsfs}
\usepackage{mathtools}
\usepackage{bbm}
\usepackage{esint}

\usepackage[margin=3cm]{geometry}
\usepackage{setspace}
\usepackage{indentfirst}
\usepackage{graphicx}
\usepackage{graphics}
\usepackage{lscape}
\usepackage{pgf,tikz}
\usepackage{tikz-cd}
\usepackage{color}
\usepackage{pict2e}
\usepackage{epic}
\usepackage{epstopdf}
\usepackage{titlesec, titlefoot}
	\titleformat{\section}[block]{\Large\bfseries\filcenter}{\thesection}{1em}{}
\usepackage{commath}
\usepackage{float}
\usepackage{caption}
\usepackage{etoolbox}
\usepackage[affil-it]{authblk}
\usepackage{combelow}

\usepackage[hidelinks, bookmarksdepth=3]{hyperref}
\hypersetup{bookmarksopen=true} 
\usepackage{hypcap}

\graphicspath{{./Pictures/}}
\allowdisplaybreaks

\expandafter\def\expandafter\normalsize\expandafter{%
    \normalsize
    \setlength\abovedisplayskip{6pt}
    \setlength\belowdisplayskip{6pt}
    \setlength\abovedisplayshortskip{6pt}
    \setlength\belowdisplayshortskip{6pt}
}

\theoremstyle{plain}

\renewcommand*\thesection{\arabic{section}}
\numberwithin{equation}{section} 

\theoremstyle{plain}
\newtheorem{thm}{Theorem}
\newtheorem{lemma}[thm]{Lemma}
\newtheorem{prop}[thm]{Proposition}
\newtheorem{cor}[thm]{Corollary}
\numberwithin{thm}{section} 

\theoremstyle{definition}

\newtheorem{remark}[thm]{Remark}
\newtheorem{conj}[thm]{Conjecture}

\newcommand{\thistheoremname}{}
\newtheorem{genericthm}[thm]{\thistheoremname}

\newcommand{\thistheoremnames}{}
\newtheorem*{genericthms}{\thistheoremnames}
\newenvironment{para*}[1]
  {\renewcommand{\thistheoremnames}{#1}%
   \begin{genericthms}}
  {\end{genericthms}}

\expandafter\let\expandafter\oldproof\csname\string\proof\endcsname
\let\oldendproof\endproof
\renewenvironment{proof}[1][\proofname]{%
  \oldproof[\upshape \bfseries #1]%
}{\oldendproof}

\makeatletter
\def\@makechapterhead#1{%
  \vspace*{50\p@}%
  {\parindent \z@ \raggedright \normalfont
    \interlinepenalty\@M
    \Huge\bfseries  \thechapter.\quad #1\par\nobreak
    \vskip 40\p@
  }}
\makeatother

\def \a{\alpha}
\def \R {\mathbb{R}}

\def \D{\textup{D}}

\def \e{\varepsilon}
\def \d{\,\textup{d}}

\def \exc{\backslash}
\def \p{\partial}

\def \mb{\mathbb}

\def \tp{\textup}
\def \Id{\textup{Id}}
\def \id{\textup{id}}

\begin{document}

\title{\textbf{Automatic quasiconvexity of homogeneous isotropic rank-one convex integrands}}

\author[1]{{\Large Andr\'e Guerra}}
\author[2]{{\Large Jan Kristensen}}

\affil[1]{\small School of Mathematics, Institute for Advanced Study, 1 Einstein Dr., Princeton, NJ, 08540, USA 
\protect \\
  {\tt{aguerra@math.ias.edu}}
  \vspace{1em} \ }

\affil[2]{\small 
University of Oxford, Andrew Wiles Building Woodstock Rd, Oxford OX2 6GG, United Kingdom 
\protect\\
  {\tt{jan.kristensen@maths.ox.ac.uk}} \ }

\date{}

\maketitle

\begin{abstract}
We consider the class of non-negative rank-one convex isotropic integrands on $\R^{n\times n}$ which are also positively $p$-homogeneous. If $p\leq n=2$ we prove, conditional on the quasiconvexity of the Burkholder integrand, that the integrands in this class are quasiconvex at conformal matrices. If $p\geq n =2$, we show that the positive part of the Burkholder integrand is polyconvex. In general, for $p\geq n$, we prove that the integrands in the above class are polyconvex at conformal matrices. Several examples imply that our results are all nearly optimal.
\end{abstract}

\unmarkedfntext{
\hspace{-0.74cm}
\emph{Acknowledgments.} A.G. was supported by the Infosys Membership at the Institute for Advanced Study. It is a pleasure to thank Kari Astala and Daniel Faraco for interesting discussions.
}

\section{Introduction}

For $n,m\geq 2$, let $\Omega\subset \R^n$  be a bounded domain and consider the functional
$$\mathscr F[u]\equiv \int_{\Omega} F(\D u)\d x, \qquad u\colon \Omega\to \R^m.$$
The integrand $F\colon \R^{m\times n}\to \R$ is assumed to be Borel measurable and locally bounded. 
A classical problem in the Calculus of Variations is to minimize $\mathscr F$ over Dirichlet classes $W^{1,p}_g(\Omega,\R^m)\equiv g+W^{1,p}_0(\Omega,\R^m)$, where $g\in W^{1,p}(\R^n,\R^m)$ is a given boundary datum. 
The only systematic way of solving this minimization problem is through the Direct Method, which requires that $\mathscr F$ be both \textit{sequentially weakly lower semicontinuous} and \textit{coercive} in $W^{1,p}_g(\Omega)$. 

The natural condition on $F$ which ensures that the Direct Method is applicable is quasiconvexity, as introduced by Morrey \cite{Morrey1952}. We say that $F$ is \textit{quasiconvex} at $A\in \R^{m\times n}$ if
\begin{equation*}
\label{eq:quasiconvexity}
F(A)\leq \frac{1}{\mathscr L^n(\Omega)}\int_{\Omega} F(A+\D\varphi) \d x \qquad \tp{ for all } \varphi \in W^{1,\infty}_0(\Omega,\R^m).
\end{equation*}
In fact, if $F$ satisfies the natural growth condition
\begin{equation*}
\label{eq:pgrowth}
|F(A)|\leq C (|A|^p+1) \quad \text{for all } A\in \R^{m\times n},
\end{equation*}
then $\mathscr F$ is sequentially weakly lower semicontinuous in $W^{1,p}_g(\Omega,\R^m)$ if and only if $F$ is quasiconvex \cite{Acerbi1984,Meyers1965,Morrey1952}, and $\mathscr F$ is coercive in $W^{1,p}_g(\Omega,\R^m)$ if and only if there is $c>0$ and $A_0\in \R^{m\times n}$ such that $F-c|\cdot|^p$ is quasiconvex at $A_0$ \cite{Chen2017}.

Quasiconvexity has been mostly studied in relation to two competing semiconvexity notions, polyconvexity and rank-one convexity: an integrand $F\colon \R^{m\times n}\to \R$ is
\begin{enumerate}
\item \textit{polyconvex} at $A\in \R^{m\times n}$ if there is a convex $G\colon \R^{\tau(m,n)}\to \R$ such that 
$F(A)=G( M(A))$ and $F\geq G\circ M,$
where $M(A)$ is the vector of minors of $A$, with  length $\tau(m,n)$;
\item \textit{rank-one convex} if $t\mapsto F(A+tX)$ is convex for all $A,X\in \R^{m\times n}$ such that $\tp{rank}(X)=1$.
\end{enumerate}It is well-known that 
$$\textup{polyconvexity} \implies \textup{quasiconvexity} \implies 
\textup{rank-one convexity},$$
although in general none of the implications can be reversed \cite{Dacorogna2007,Muller1999a}. Nonetheless, there are very few examples of rank-one convex non-quasiconvex integrands: the first such example was constructed by \v Sver\'ak \cite{Sverak1992a}, with $m\geq 3, n\geq 2$, and more recently another example was found by Grabovsky \cite{Grabovsky2018}, with $m\geq 8, n\geq 2$. In particular, there are no such examples when $m=2$, and there is some evidence that in this case rank-one convexity may imply quasiconvexity \cite{Astala2012, Faraco2008,GuerraCosta2020,
  Harris2018,Kirchheim2008, Muller1999b,Muller2000,
  Pedregal1996,Pedregal1998,Sebestyen2017}. In this paper we will specially focus on the case $m=n$, but see also Remark \ref{rmk:rectangular} for the general case.

A separate but related problem is whether there are structural conditions on rank-one convex integrands which ensure quasiconvexity. An interesting structural condition is that of \textit{positive $p$-homogeneity}, which is to say that
\begin{equation}
\label{eq:hom}
F(tA)=t^p F(A) \quad \text{for all }A\in \R^{n\times n} \text{ and all } t>0.
\end{equation}
The second author showed, in joint work with Kirchheim \cite{Kirchheim2016}, that a positively $1$-homogeneous rank-one convex integrand is necessarily \textit{convex} at all matrices with rank at most one. Moreover, this result is sharp in the sense that there exist integrands with those properties which are not convex at matrices with rank two. We refer the reader to \cite{Dacorogna2008,Muller1992} for related results.

In this paper we investigate whether a result with a similar flavour exists for $p>1$. Positive homogeneity is not per se enough to upgrade rank-one convexity to quasiconvexity: for instance,  Grabovsky’s example is positively 2-homogeneous. It is for this reason that we will restrict our attention to integrands which are  \textit{isotropic}, in the sense that
\begin{equation}
\label{eq:isotropy}
F(QAR)=F(A) \quad 
\tp{ for all } A\in \R^{n\times n} \tp{ and all } Q, R\in \tp{SO}(n).
\end{equation}
Isotropy is a natural condition often satisfied by the integrands of interest in Nonlinear Elasticity \cite{Ball1977,Ball1982} and in Geometric Function Theory \cite{Iwaniec2001}; moreover, neither \v Sver\'ak's nor Grabovsky's examples are isotropic.  

Zero is a distinguished point for positively homogeneous integrands. In this paper we will see that, for isotropic integrands, \textit{conformal matrices} are also distinguished. We define
$$\textup{CO}(n)\equiv \{tQ:Q\in \tp{O}(n), t\in \R\},$$
which splits into the sets of conformal and anti-conformal matrices:
\begin{equation}
\label{eq:conformalmatrices}
\textup{CO}^+(n)\equiv \tp{CO}(n)\cap\{\det\geq 0\}, \qquad \textup{CO}^-(n)\equiv \tp{CO}(n)\cap\{\det\leq 0\}.
\end{equation}
We note that the sets $\textup{CO}^\pm(n)$ are \textit{nonlinear} cones, except when $n=2$.

Arguably the most important rank-one convex integrand satisfying both \eqref{eq:hom} and \eqref{eq:isotropy} is the \textit{Burkholder integrand}, $B_p\colon \R^{n\times n}\to \R$, defined by 
\begin{equation}
\label{eq:defBp}
B_p(A)=\left(\Big|1-\frac np\Big| |A|^n + \det A\right)|A|^{p-n}, \qquad p\geq \frac n 2.
\end{equation}
Here and throughout the paper we write $|A|$ for the operator norm of $A\in \R^{n\times n}$.
Observe that $B_n(A)=\det A$; for a general $p$, $B_p$ can be thought of as a positively $p$-homogeneous version of the determinant.
The two-dimensional version of $B_p$ was found by Burkholder \cite{Burkholder1984,Burkholder1988} in the context of martingale theory and the generalization to higher dimensions that we consider here is due to Iwaniec \cite{Iwaniec2002}. A related example was found by Dacorogna and Marcellini \cite{Alibert1992,Dacorogna1988}, see also \cite{Iwaniec2005}. The following is a long-standing open problem:

\begin{conj}[Iwaniec]
\label{conj:Iwaniec}
The integrand $B_p\colon \R^{2\times 2}\to \R$ is quasiconvex.
\end{conj}

Quasiconvexity of $B_p$ is closely related to sharp integrability properties of quasiconformal mappings, see for instance \cite{Astala2009,Astala2012,Iwaniec2002} for a wealth of information on this connection. As remarked by the first author in \cite{Guerra2018}, to prove Conjecture \ref{conj:Iwaniec} it suffices to show that $B_p$ is \textit{quasiconvex at a single point}.
After the beautiful results of Astala--Iwaniec--Prause--Saksman \cite{Astala2012}, see also \cite{Banuelos2008}, there is now ample evidence to believe Conjecture \ref{conj:Iwaniec}.

In applications, the integrands of interest are often \textit{non-negative},
c.f.\ \cite{Rosakis1994} for similar considerations, and this is the final structural assumption that we will make.
 By considering the positive part of any rank-one convex integrand, such as
$$B_p^+\equiv \max\{0,B_p\},$$
we obtain plenty of interesting non-negative rank-one convex integrands. Another source of examples are envelopes of distance functions:  one considers
$(\tp{dist}_K^p)^\tp{rc},$
where $K\subset \R^{n\times n}$ is an isotropic cone and the superscript denotes the rank-one convex envelope, see  \cite{Faraco2004,Sverak1991,Zhang1992,Zhang1997b}. Yet other examples are obtained from conformal or nearly-conformal energy functions, as studied in \cite{Iwaniec1993,Iwaniec1996,Muller1999f,Yan1997,Yan1998a,Yan2001a}.

Our main theorem shows that Conjecture \ref{conj:Iwaniec} is the key to establishing quasiconvexity of the class of integrands under consideration:

\begin{thm}
\label{thm:automaticqc}
Let $F\colon \R^{2\times 2}\to[0,+\infty)$ be an isotropic rank-one convex integrand which is positively $p$-homogeneous for some $p>1$. Suppose that 
$$\tp{either } p\leq 2 \text{ and Conjecture \ref{conj:Iwaniec} holds} \qquad \text{or } p\geq 2.$$
Then $F$ is quasiconvex in $\textup{CO}(2)$.
\end{thm}

The case $p\leq 2$ is a generalization of earlier results in \cite{Astala2015}, while the case $p\geq 2$ is deduced from the following stronger result:

\begin{thm}
\label{thm:automaticpc}
Let $F\colon \R^{n\times n}\to [0,+\infty)$ be an isotropic rank-one convex integrand which is positively $p$-homogeneous for  some $p\geq n$.  $F$ is polyconvex in $\textup{CO}(n)$.
\end{thm}

The conclusion of Theorem \ref{thm:automaticpc} fails if any of its assumptions is removed:
\begin{remark}[Necessity of all conditions]\label{rmk:optimality}
We take $n=2$.
\begin{enumerate}
\item \label{it:necnonneg} Necessity of non-negativity: for $p>2$, the integrand $B_p$ is rank-one convex, isotropic, positively $p$-homogeneous, but non polyconvex at $\Id$.
\item\label{it:necp>n} Necessity of $p\geq n$: for $p<n=2$, the integrand $B_p^+$ is rank-one convex, isotropic, non-negative, but non polyconvex at $\Id$.
\item\label{it:neciso} Necessity of isotropy: for $p\geq 2$ there is an integrand $F_p\colon \R^{2\times 2}\to[0,+\infty)$ which is rank-one convex, positively $p$-homogeneous but which is not polyconvex at $\Id$.
\end{enumerate}
We refer the reader to Propositions \ref{prop:necisotropy} and \ref{prop:nonpc} for proofs of the above claims and to \cite{Sylhavy2002,Sverak1991} for an example which shows the necessity of homogeneity assumptions.
\end{remark}

The proof of Theorem \ref{thm:automaticpc} is an application of the well-known Baker--Ericksen inequalities, which are satisfied by all rank-one convex isotropic integrands. There is a very extensive literature on isotropic rank-one convex integrands, see e.g.\ \cite{Aubert1995,Ball1984a,Dacorogna2001,Dacorogna1999,Knowles1976,Mielke2005} and the references therein, as well as the papers by \v Silhavy \cite{Silhavy1999,Silhavy2002,Silhavy2003}. Here we give a new proof of the fact that the Baker--Ericksen inequalities imply a certain monotonicity property; although this fact had been previously observed in the literature, the connection with homogeneous integrands had not been made.

We also show that the integrands in Theorem \ref{thm:automaticpc} are in general not polyconvex everywhere:

\begin{prop}\label{prop:Bpn=3}
For each $p>3$, the integrand $B_p^+\colon \R^{3\times 3}\to [0,+\infty)$ is not polyconvex.
\end{prop}

When combined with Remark \ref{rmk:optimality}, Proposition \ref{prop:Bpn=3} shows that Theorem \ref{thm:automaticpc} is, in general, optimal. 
Nonetheless, we are yet to address the two-dimensional case, and here we have a surprise for the reader:

\begin{prop}\label{prop:Bpn=2}
The integrand $B_p^+\colon \R^{2\times 2}\to [0,+\infty)$ is polyconvex for $p\geq 2$.
\end{prop}

In fact, the situation appears to be genuinely different when $n=2$, and it may well be that in this case all the integrands satisfying the conditions of Theorem \ref{thm:automaticpc} are polyconvex: we have not been able to find any counter-example to this claim. There are  two other results related to this possibility: the first is that any connected, compact isotropic set $K\subset \R^{2\times 2}$ that is rank-one convex is also polyconvex \cite{Cardaliaguet2002a,Cardaliaguet2002,Conti2003}; the second is that rank-one convexity implies polyconvexity for isotropic integrands in $\R^{2\times 2}_+$ which are $0$-homogeneous \cite{Martin2017}. In any case, polyconvexity is a nonlocal condition \cite{Kristensen2000}, which makes it difficult to verify.

%

In the above statements, we have focused on the case $m=n$; as promised above, we now indicate how our results extend to the general case.

\begin{remark}[Rectangular matrices]
\label{rmk:rectangular}
In general we can replace the isotropy condition \eqref{eq:isotropy} with the assumption that $F$ is $\textup{SO}(m)\times \textup{SO}(n)$-invariant. It turns out that, when $m\neq n$, this is equivalent to $F$ being $\textup{O}(m)\times \textup{O}(n)$-invariant \cite[Proposition 5.32]{Dacorogna2007}. Integrands which are $\textup{O}(m)\times \textup{O}(n)$-invariant can be written as functions of the singular values, while \eqref{eq:isotropy} only says that $F$ can be written as a function of the \textit{signed} singular values, see already Lemma \ref{lemma:isotropy}. Thus, due to extra symmetry, the case $m\neq n$ is much simpler to treat. We refer the reader to Remark \ref{rmk:rectangularextension} for the corresponding generalization of Theorem \ref{thm:automaticpc}.
\end{remark}

We conclude this introduction with a brief outline of the paper. Section \ref{sec:notation} contains useful notation and some preliminary remarks about $B_p$. Section \ref{sec:p>n} contains a general discussion of isotropic rank-one convex integrands, together with the proof of Theorem \ref{thm:automaticpc}. Section \ref{sec:Bp} contains the proof of Theorem \ref{thm:automaticqc} and of Propositions \ref{prop:Bpn=3} and \ref{prop:Bpn=2}.

\section{Notation}\label{sec:notation}

For an integrand $F\colon \R^{n\times n}\to\R$ we write $F^+\equiv \max\{F,0\}$.
We denote by $\Id\in \R^{n\times n}$ the identity matrix and $\overline{\Id}\equiv \tp{diag}(1,\dots, 1,-1)$. 
For a matrix $A\in \R^{n\times n}$, we write 
$$|A|\equiv \max_{|\xi|=1} |A\xi|$$ for its operator norm. We also write $\R^{n\times n}_\pm\equiv \R^{n\times n}\cap \{\pm \det>0\}$.

Recall the definition \eqref{eq:conformalmatrices} of the sets of conformal and anti-conformal matrices.
From the identity $\tp{cof}(A) A^\tp{T}=(\det A )\,\Id$, it is easy to see that $A\in \R^{n\times n}$ is conformal if and only if $\tp{cof}(A)=|A|^{n-2} A$, where $\tp{cof}(A)$ is the cofactor matrix of $A$.
This leads us to the definition of the conformal and anti-conformal parts of a matrix \cite{Iwaniec2002}:
\begin{equation*}
\label{eq:defconformalpart}
A^\pm \equiv \frac 1 2 \left(|A|^{\frac{n-2}{2}} A \pm |A|^{\frac{2-n}{2}}\tp{cof}(A)\right).
\end{equation*}
Our choice of powers of $|A|$ in this definition is so that the forms $|A^\pm|$ are $\frac n 2$-homogeneous, as $\frac n 2$ is the critical Sobolev exponent for the regularity theory of conformal mappings \cite{Iwaniec1993b}, and
\begin{equation}
\label{eq:normconformalpart}
|A^\pm|=\frac 1 2\Big(|A|^{\frac n 2} \pm |A|^{-\frac n 2} \det A\Big),
\end{equation}
c.f.\ \cite[Lemma 2.1]{Iwaniec2002}.
This yields the useful identities
\begin{equation}
\label{eq:conformalcoords}
|A|^{\frac n 2}=|A^+|+|A^-|, \qquad \det A = |A^+|^2-|A^-|^2.
\end{equation}

Comparing \eqref{eq:defBp} with \eqref{eq:normconformalpart}, we observe that
\begin{equation*}
\label{eq:Bn/2}
B_{n/2}(A)=2 |A^+ |.
\end{equation*}
Although we will not use this fact in the sequel, it is conceptually helpful to realize that, for $p>\frac n2$, $B_p^+$ is a generalization of $|A^+|$ which measures the distance of $A$ to the set of anti $\frac{p}{|p-n|}$-quasiconformal matrices. To make this precise, we introduce the cones 
$$\tp{QCO}^\pm(n,K)\equiv \left\{A\in \R^{n\times n}: |A|^n\leq \pm K \det A\right\}$$
for $K\geq 1$. In particular, $\tp{QCO}^\pm(n,1)=\tp{CO}^\pm(n)$. Now simply note that
$B_p(A)\leq 0$ if and only if $A\in \tp{QCO}^-(n,K)$, where $K\equiv \frac{p}{|p-n|}.$


\section{Isotropic integrands and automatic polyconvexity}\label{sec:p>n}

Given $A\in \R^{n\times n}$, we write $\lambda(A)\equiv(\lambda_1(A),\dots, \lambda_n(A))$ for the vector of \textit{signed singular values} of $A$, which is defined as follows: the numbers 
$$\lambda_1(A)\geq \dots \geq \lambda_{n-1}(A)\geq |\lambda_n(A)|$$ are the eigenvalues of the positive-definite matrix $\sqrt{A^\tp{T} A}$, and $\textup{sign}(\lambda_n(A))=\textup{sign}(\det A)$. 

It is well-known that isotropic integrands admit a simple representation in terms of the signed singular values \cite{Dacorogna2007}:

\begin{lemma}
\label{lemma:isotropy}
An integrand $F\colon \R^{n\times n}\to \R$ is isotropic if and only if there is a symmetric, even function $f\colon \R^n\to \R$ such that
\begin{equation}
\label{eq:repisotropy}
F(A)=f(\lambda_1(A),\dots, \lambda_n(A)).
\end{equation}
\end{lemma}

We recall that $f\colon \R^n\to \R$ is \textit{symmetric} if $f(P\lambda)=f(\lambda)$ for every $\lambda\in \R^n$ and every permutation matrix $P\in \R^{n\times n}$, and $f$ is \textit{even} if $f(\sigma \lambda)=f(\lambda)$ whenever $\sigma\in \textup{SO}(n)$ is a diagonal matrix.

By restricting $F$ to the diagonal matrices it is clear that if $F$ is $C^k$ then so is $f$, for any $k=0,1,\dots, \infty$. Remarkably, the much deeper converse is also true \cite{Ball1984a,Silhavy2000}.

\subsection{The Baker--Ericksen inequalities and monotonicity}

We begin by recalling the well-known Baker--Ericksen inequalities:

\begin{prop}[Baker--Ericksen inequalities]\label{prop:BE}
Let $F\colon \R^{n\times n}_+\to \R$ be a $C^1$ isotropic rank-one convex integrand with representation \eqref{eq:repisotropy}. If $i\neq j$ and $\lambda_i\neq \lambda_j$ then
$$\frac{\lambda_i \p_i f(\lambda) - \lambda_j \p_j f(\lambda)}{\lambda_i-\lambda_j}\geq 0$$
for all $\lambda\in \R^n$ such that $\lambda_k>0$ for all $k=1,\dots, n$.
\end{prop}

We include a proof for the sake of completeness.

\begin{proof}
Without loss of generality we may take $i=1, j=2$ and $n=2$. Fix $\lambda_1>\lambda_2>0$ and consider the line parametrized by
$$A(t)=\begin{bmatrix}
\sqrt{\lambda_1\lambda_2} & t \\ 0 & \sqrt{\lambda_1\lambda_2}
\end{bmatrix}.$$
The signed singular values $\mu_1(t)\geq \mu_2(t)>0$ of $A(t)$ satisfy
$$
\begin{cases}
\mu_1(t)^2+\mu_2(t)^2= 2 \lambda_1\lambda_2 + t^2,\\ \qquad \mu_1(t)\mu_2(t)=\lambda_1\lambda_2,
\end{cases}
\quad \implies \quad \dot \mu_i(t)= \frac{t \mu_i(t)}{\mu_1(t)^2-\mu_2(t)^2};
$$
in fact, they are the unique solution of the above system of equations.
With $\varphi(t)=F(A(t))$,
$$\dot \varphi(t) = \dot \mu_1(t)\p_1 f+\dot \mu_2(t) \p_2 f = t \frac{\mu_1(t)\p_1 f- \mu_2(t) \p_2 f}{\mu_1(t)^2-\mu_2(t)^2}.$$
Since the line $\{A(t):t\in \R\}$ is parallel to a rank-one matrix, $\varphi\colon \R\to \R$ is convex; moreover, as $\mu_i(t)=\mu_i(-t)$, $\varphi$ is also even. It follows that $\varphi$ is non-decreasing on $[0,+\infty)$ and so $\dot\varphi(t)\geq 0$ for $t>0$.
Taking $t=\lambda_1-\lambda_2>0$ we obtain $\mu_i(t)=\lambda_i$ and the conclusion follows.
\end{proof}

We next show that the Baker--Ericksen inequalities imply a certain monotonicity property of isotropic rank-one convex integrands.

\begin{prop}[Monotonicity]\label{prop:monotonicity}
Let $F\colon \R^{n\times n}\to \R$ be rank-one convex and isotropic. If $A,B\in \R^{n\times n}$ are such that 
\begin{equation}
\label{eq:eigenvalscond}
\prod_{i=1}^k \lambda_i(A) \leq \prod_{i=1}^k \lambda_i(B) \quad \text{for } k=1,\dots, n-1, \qquad \prod_{i=1}^n \lambda_i(A) = \prod_{i=1}^n \lambda_i(B)
\end{equation}
then $F(A)\leq F(B)$.
\end{prop}

An analogue of Proposition \ref{prop:monotonicity} was discovered by Dacorogna and Marcellini \cite[\S 7.3]{Dacorogna1999} whenever $F$ is $\tp{O}(n)$-invariant; in this case, the second condition in \eqref{eq:eigenvalscond} can be omitted.  Proposition \ref{prop:monotonicity} was  proved by  \v Silhavy in \cite{Silhavy1999} and here we give a more direct proof.

\begin{proof}
By mollifying $F$ we can clearly assume that it is smooth. Consider the domains 
\begin{align*}
&\mb G_n\equiv \{\lambda \in \R^n: \lambda_1>\lambda_2>\dots>\lambda_n>0\},\\
&\mb D_n\equiv \{\sigma\in \R^n: \sigma_1> 0, (\sigma_1)^i>\sigma_i>0 \text{ for } i=2, \dots, n.\}.
\end{align*}
Taking $\sigma_i(\lambda)=\prod_{j=1}^i \lambda_i$, consider the map $\psi\colon\mb G_n\to \mb D_n$ defined by
$$\psi(\lambda_1,\dots, \lambda_n)\equiv (\sigma_1(\lambda), \sigma_2(\lambda), \dots, \sigma_n(\lambda)).$$
It is easy to see that $\psi\colon \mb G_n\to \mb D_n$ is a smooth diffeomorphism.  Let $f$ be the representation of $F$ provided by Lemma \ref{lemma:isotropy} and consider the function $g\equiv f\circ \psi^{-1}\colon \mb D_n\to \R$, i.e.\ 
$$g(\sigma_1(\lambda), \sigma_2(\lambda), \dots, \sigma_n(\lambda))=f(\lambda_1, \lambda_2,\dots, \lambda_n).$$
Using the chain rule, we calculate
$$
\p_{\lambda_i} f=\sum_{j=i}^n\left(\prod_{k=1,k\neq i}^{j} \lambda_k \right)\p_j g$$
and so, for any $1\leq i<n$, according to Proposition \ref{prop:BE},
$$0\leq \lambda_{i} \p_{\lambda_i} f -\lambda_{i+1} \p_{\lambda_{i+1}} f = \left(\prod_{k=1}^{i} \lambda_k \right) \p_i g.$$
Thus $g$ is non-decreasing separately in its arguments, except maybe in the last one.

Suppose now that $A,B\in \R^{n\times n}_+$ satisfy \eqref{eq:eigenvalscond} and are such that $\lambda(A),\lambda(B)\in \mb G_n$. Writing $\sigma_i(A)\equiv \sigma_i(\lambda(A))$ as an abuse of notation, we have
\begin{align*}
F(A)=f(\lambda(B))
&= g(\sigma_1(B),\dots, \sigma_n(B))
\\ 
&=g\left(\sigma_1(B),\dots,\sigma_{n-1}(B), \sigma_n(A)\right) \\
& \geq g\left(\sigma_1(B),\dots,\sigma_{n-1}(A), \sigma_n(A)\right)\\
&\geq \dots \\
&\geq g(\sigma_1(A),\dots, \sigma_n(A))=f(\lambda(A))=F(B).
\end{align*}
Note that, for any $i$, $\lambda_1(B)^i\geq \lambda_1(A)^i> \sigma_i(A)$, so we always evaluate $g$ on its domain $\mb D_n$. Since the set of matrices in $\R^{n\times n}_+$ such that $\lambda(A),\lambda(B)\in \mb G_n$ is dense in $\R^{n\times n}_+$, the conclusion follows from continuity of $F$. 

When $\det A=\det B=0$ we again argue by continuity. Let $k=\max\{i: \lambda_i(A)>0\}<n$ and, for $\e>0$, consider the matrices $A_\e,B_\e\in \R^{n\times n}_+$ defined by
$$A_\e\equiv \tp{diag}(\lambda_1(A),\dots, \lambda_k(A), \e,\dots, \e), \qquad B_\e\equiv \tp{diag}(\lambda_1(B),\dots, \lambda_k(B), \e ,\dots, \e, \e \zeta),$$
where $\zeta \equiv \prod_{i=1}^k \lambda_i(A)/\lambda_i(B)$, which clearly satisfy \eqref{eq:eigenvalscond}. Since $F(A_\e)\leq F(B_\e)$ the conclusion follows by letting $\e\to 0$.

To deal with the case where $\det A=\det B<0$, we  replace $F$ with $\tilde F(A)\equiv F(\overline \Id A)$. The integrand $\tilde F$ is easily seen to be rank-one convex and isotropic, with representation $\tilde F(A)=f(\lambda_1(A),\dots, \lambda_{n-1}(A),-\lambda_n(A))$. Thus $F(A)=\tilde F(\overline \Id A)\leq \tilde F(\overline \Id B)=F(B)$, as wished.
\end{proof}

\subsection{Proof of Theorem \ref{thm:automaticpc} and other consequences}

As an immediate consequence of Proposition \ref{prop:monotonicity}, we obtain the following:

\begin{cor}\label{cor:infCO(n)}
Let $F\colon \R^{n\times n}\to \R$ be an isotropic rank-one convex integrand. Then
$$\inf_{\R^{n\times n}} F = \inf_{\tp{CO}(n)} F.$$
\end{cor}

\begin{proof}
Let $B\in \R^{n\times n}_+$ and note that $A=(\det B)^{\frac 1 n} \Id\in \tp{CO}^+(n)$ satisfies \eqref{eq:eigenvalscond}, thus
$$F((\det B)^{\frac 1 n} \Id )\leq F(B) \quad \implies \quad \inf_{\R^{n\times n}_+} F = \inf_{\tp{CO}^+(n)} F.$$
By replacing $F$ with $\tilde F$ as above, we deduce a similar identity in $\R^{n\times n}_-$. The case where $\det B=0$ is obtained by continuity.
\end{proof}

Combining Corollary \ref{cor:infCO(n)} with the results of Yan \cite{Yan1997} we obtain the following result:

\begin{cor}\label{cor:p>n/2}
Let $F\colon \R^{n\times n}\to [0,+\infty)$ be an isotropic rank-one convex integrand which is positively $p$-homogeneous. Then
$$F(A_0)=0 \text{ for some } A_0\neq 0 \quad \implies \quad p\geq \frac n 2.$$
\end{cor}

The exponent $\frac n 2$ is sharp, since the integrands $A\mapsto |A^\pm|$ vanish exactly on $\tp{CO}^\mp(n)$.

\begin{proof}
If $F(A_0)=0$ for some $A_0\neq 0$ it follows that $F$ vanishes on either $\tp{CO}^+(n)$ or $\tp{CO}^-(n)$, so by \cite{Yan1997} we must have that $F$ is identically zero.
\end{proof}

Finally, by arguing as in Corollary \ref{cor:infCO(n)}, we can prove Theorem \ref{thm:automaticpc}.

\begin{proof}[Proof of Theorem \ref{thm:automaticpc}]
Since $F$ is positively $p$-homogeneous, by Proposition \ref{prop:monotonicity}, we have
$$F(\Id) (\tp{det}^+)^{\frac p n}\leq F \text{ in } \R^{n\times n}_+.$$
As $F\geq 0$ and $\det^+$ vanishes in the complement of $\R^{n\times n}_+$, we see that this inequality in fact holds everywhere. The function $F(\Id) (\tp{det}^+)^{\frac p n}$ is polyconvex, it is below $F$ and it agrees with $F$ at $\Id$, so $F$ is polyconvex at $\Id$. By homogeneity and isotropy, it follows that $F$ is polyconvex in $\tp{CO}^+(n)$. A similar argument shows that $F$ is also polyconvex in $\tp{CO}^-(n)$.
\end{proof}

\begin{remark}[Rectangular matrices]
\label{rmk:rectangularextension}
As mentioned in Remark \ref{rmk:rectangular}, Theorem \ref{thm:automaticpc} admits an extension to the case of integrands $F\colon \R^{m\times n}\to \R$ where $m\neq n$. More precisely, the following holds:  an $\textup{SO}(m)\times \textup{SO}(n)$-invariant rank-one convex integrand $F\colon \R^{m\times n}\to [0,+\infty)$, which is also positively $p$-homogeneous for some $p\geq \min\{m,n\}$, is polyconvex in the set
$$\left\{A\in \R^{m\times n}: \mu(A)=t\e \text{ for some } t\in \R \text{ and } \e\in \{-1,1\}^{\min\{m,n\}}\right\},$$
where $\mu(A)$ is the vector of singular values of $A$.
By \cite[Propositions 5.31 and 5.32]{Dacorogna2007}, if $m>n$, $F$ is $\tp{SO}(m)\times \tp{SO}(n)$ invariant if and only if there is a symmetric, fully even function $f\colon \R^n\to \R$ such that
$$F(A)=f(\mu_1(A),\dots, \mu_n(A));$$
the case $m<n$ is similar. Recall that $f\colon \R^n \to \R$ is fully even if $f(\sigma \lambda)=f(\lambda)$ for any $\lambda\in \R^n$ and any diagonal matrix $\sigma \in \tp{O}(n)$.
Using this representation we may repeat the arguments of this section to prove the above statement, and we leave the details to the interested reader.
\end{remark}

\subsection{Necessity of isotropy}

The purpose of this subsection is to give an example which shows that isotropy is a necessary condition for Theorem \ref{thm:automaticpc}. The next proposition proves, in particular, the claim in Remark \ref{rmk:optimality}\ref{it:neciso}.

\begin{prop}\label{prop:necisotropy}
For $p\geq 2$, there is an integrand $F_p\colon \R^{2\times 2}\to [0,+\infty)$ which is rank-one convex, positively $p$-homogeneous and non-polyconvex at $\Id$. 
\end{prop}

\begin{proof}
We define the integrand $F_p\colon \R^{2\times 2}\to [0,+\infty)$ through
$$F_p(A)\equiv \left(\max\{a,0\}\max\{d,0\} + \frac 1 2 (b^2+c^2)\right)^{p/2},\qquad \text{where } A=\begin{bmatrix}
a & b \\ c & d
\end{bmatrix}.$$
Clearly $F_p$ is positively $p$-homogeneous. 
We first show that $F_2$ is rank-one convex.  Consider the integrands $G,H\colon \R^{2\times 2} \to \R$ defined by
$$G(A)\equiv a d + \frac 1 2(b^2+c^2), \qquad H(A)\equiv \frac 1 2 (b^2+c^2).$$
Clearly $H$ is convex. A simple calculation reveals that $G$ is rank-one convex: indeed, with $x=(x_1,x_2)$ and $y=(y_1,y_2)$, we have
$$\frac{\d^2 G(A+t x\otimes y)}{\d t^2} = 2 x_1 x_2 y_1 y_2 + x_1^2 y_2^2+x_2^2 y_1^2=(x_1y_2+x_2y_1)^2\geq 0.$$
Consider the set $U\equiv \{A\in \R^{2\times 2}: a,d>0\}$ and note that
$$F_2 = \begin{cases}
G & \text{in }  U,\\
H & \text{in } \R^{2\times2}\exc U.
\end{cases}$$
The set $U$ is a  component of $\{G>H\}$, so it follows from  \cite[Lemma 3.1]{Sverak1990} that $F_2$ is rank-one convex.
The general case $p>2$ now follows easily: the function $T_p(s)= s^{p/2}$ is convex and non-decreasing in $[0,+\infty)$ and since $F_2$ is rank-one convex and non-negative, the integrand  $F_p=T_p\circ F_2$ is rank-one convex as well.

We now prove that $F_p$ is not polyconvex at $\Id$. To see this, we consider the points
$$A_1\equiv\tp{diag}(-3,-3), \qquad A_2\equiv\tp{diag}(9,-3),\qquad A_3\equiv\tp{diag}(-3,9),$$
which satisfy the so-called minors relations
\begin{gather*}
\Id = \frac 1 3\left(A_1+A_2+A_3\right),\qquad 
\det\Id = \frac 1 3\left(\det A_1+\det A_2 + \det A_3\right).
\end{gather*}
Were $F_p$ to be polyconvex at $\Id$, we would have
$$F_p(\Id)\leq \frac 1 3\left(F_p(A_1)+F_p(A_2)+F_p(A_3) \right)$$
see e.g.\ \cite{Dacorogna2007}. However, $F_p(\Id)=1$, while the right-hand side vanishes.
\end{proof}

The integrand in Proposition \ref{prop:necisotropy} is essentially an extension of the integrand 
$$\tp{det}^{++}(A)\equiv \det A\,\mathbbm 1_{\{A \text{ is positive definite\}}}$$
from the diagonal plane to the full space $\R^{2\times 2}$. We note that $\det^{++}$ is in fact quasiconvex in $\R^{2\times 2}_\tp{sym}$ \cite{Sverak1992} and that the method we use to extend a rank-one convex integrand from a subspace  is standard \cite{Muller2003,Sverak1990,Sverak1992a}.

\section{Around the Burkholder integrand}\label{sec:Bp}

This section contains several results concerning the Burkholder integrand. We will prove, in particular, Theorem \ref{thm:automaticqc} and Propositions \ref{prop:Bpn=3} and \ref{prop:Bpn=2}.
We begin by stating some basic properties of the Burkholder integrand, which are proved in \cite{Iwaniec2002}:

\begin{prop}[Properties of the Burkholder integrand]
\label{prop:Bpprops}
For $p\in [\frac n 2, +\infty)$, we have:
\begin{enumerate}
\item \label{it:homiso}
$B_p$ is positively $p$-homogeneous and isotropic;
\item $B_p\colon \R^{n\times n} \to \R$ is rank-one convex;
\item the set of points where $B_p$ is smooth is the dense, open set
\begin{equation*}
\label{eq:smoothpoints}
\mathcal R(n)\equiv \{A\in \R^{n\times n}: |A| ^2\tp{ is a simple eigenvalue of } A^\tp{T} A\}.
\end{equation*}
\end{enumerate}
\end{prop}
Note that $\tp{CO}(n)\subseteq \R^{n\times n}\exc \mathcal R(n)$, with equality when $n=2$.

\subsection{Extremality}

The purpose of this section is to prove the following theorem, which generalizes the results of Astala--Iwaniec--Prause--Saksman \cite{Astala2015}:

\begin{thm}\label{thm:extremality}
Let $F\colon \R^{n\times n}\to \R$ be a rank-one convex integrand which is  positively $p$-homogeneous for some $p\geq \frac n2 $ and has the form
\begin{equation}
\label{eq:symmetry}
F(A)=f(|A|^n,\det A) \qquad \text{for all } A\in \R^{n\times n} \text{ and some } f\colon \R^2\to \R.
\end{equation}
Suppose that 
$$\text{either } p\leq n\text{ and } F(\Id)=B_p(\Id)\qquad \text{  or }p\geq n\text{ and }F(\overline \Id)=B_p(\overline \Id).$$ In either case $F\geq B_p$.
\end{thm}


Before proceeding further we combine Theorems \ref{thm:automaticpc} and \ref{thm:extremality} to prove Theorem \ref{thm:automaticqc}. In order to apply Theorem \ref{thm:extremality} note that, when $n=2$, condition \eqref{eq:symmetry} is \textit{equivalent} to $F$ being isotropic, c.f.\ Lemma \ref{lemma:isotropy}, and moreover the assumption $p\geq \frac n 2=1$ poses no restriction. In arbitrary dimension, condition \eqref{eq:symmetry} is strictly stronger than isotropy. Finally the assumption $p\geq \frac n 2$ holds whenever the integrand vanishes somewhere by Corollary \ref{cor:p>n/2}.

\begin{proof}[Proof of Theorem \ref{thm:automaticqc}]

Suppose that $F$ is not identically zero, as else there is nothing to prove. We claim that
\begin{equation}
\label{eq:claimthmautomaticqc}
F(\Id)>0 \text{ or } F(\overline \Id)>0.
\end{equation}
Since $F$ is isotropic and positively homogeneous, it is uniquely determined by its values on the segment $\ell=[\Id,\overline \Id]$, on which it is convex, since $\ell$ is parallel to a rank-one matrix. It follows that $\max_\ell F= \max\{F(\Id), F(\overline \Id)\}$. Thus, if \eqref{eq:claimthmautomaticqc} did not hold, we would have $F< 0$ on $\R^{2\times 2}\exc\{0\}$, which is impossible: for any $X\in \R^{2\times 2}$ with $\tp{rank}(X)=1$, by rank-one convexity we have
$$0=F(0)\leq \frac 1 2 (F(X)+F(-X))$$
and so $F$ is non-negative either at $X$ or at $-X$. Hence \eqref{eq:claimthmautomaticqc} holds.

By replacing $F$ with $\tilde F\equiv F(\overline \Id \cdot)$ if need be we see that, without loss of generality, we may always assume that $F(\Id)>0$. Indeed, $\tilde F$ clearly inherits the rank-one convexity, isotropy and positive homogeneity of $F$, and $\tilde F(\Id)=F(\overline \Id)$.

Let us first suppose that $p\leq 2$. By the previous paragraph there is $c>0$ such that $c\,F(\Id)=B_p(\Id)$.  Since $c>0$, $cF$ is rank-one convex and hence  by Theorem \ref{thm:extremality} it is always above $B_p$. Thus, assuming that Conjecture \ref{conj:Iwaniec} holds,
$$c\,F(\Id ) = B_p(\Id) \leq \fint_\Omega B_p(\Id + \D \varphi) \d x \leq c\fint_\Omega F(\Id + \D \varphi) \d x $$ 
for any bounded Lipschitz domain $\Omega\subset \R^2$ and any $\varphi \in W^{1,\infty}_0(\Omega,\R^2)$. Thus $F$ is quasiconvex at $\Id$ and, by homogeneity, it is quasiconvex in $\tp{CO}^+(2)$. To prove that $F$ is quasiconvex in $\tp{CO}^-(2)$ we use the fact that $F\geq 0$ by assumption: if $F(\overline \Id)=0$ there is nothing to prove and, if $F(\overline \Id)>0$, we may replace $F$ with $\tilde F$ as above to conclude.

Finally, the case $p\geq 2$ is an immediate consequence of Theorem \ref{thm:automaticpc} and the remarks made just before the proof.
\end{proof}

%
%
%
%

We now proceed to the proof of Theorem \ref{thm:extremality}. We begin by recalling a few classical considerations concerning rank-one convex integrands and radial stretchings.
To be precise, a \textit{radial stretching} is a map $\phi \colon \mb B^n\to \R^n$ of the form
\begin{equation}
\label{eq:defphi}
\phi(x)=\rho(r)\frac{x}{r},\qquad r\equiv |x|,
\end{equation}  where $\rho\colon[0,1]\to[0,1]$ is a Lipschitz continuous function such that $\rho(0)=0$ and $\rho(1)=1$. Here $\mb B^n$ denotes the unit ball in $\R^n$. In particular, $\phi=\id$ on $\mb S^{n-1}$.
For a radial stretching as in \eqref{eq:defphi}, we have the identities
\begin{align}
\begin{split}
\label{eq:derivativesphi}
\det \D \phi(x)=\dot \rho(r)\frac{\rho(r)}{r}, \qquad |\D \phi(x)|=\max\Big\{|\dot \rho(r)|, \frac{\rho(r)}{r}\Big\}.
\end{split}
\end{align}
We will also use the \textit{conjugate mapping} $$\bar \phi(x)\equiv \rho(r) \frac{\bar x}{r},\qquad \bar x\equiv (x_1,\dots, x_{n-1},-x_n),$$
which satisfies $\det \D \bar \phi=-\det \D \phi$ and $|\D \bar \phi|=|\D\phi|$. 

We recall that, in general, rank-one convex integrands are quasiconvex along radial stretchings \cite{Ball1990,Sivaloganathan1988}. A very direct proof of the next proposition can be found in \cite[Proposition 3.4]{Ball1990a}:

\begin{prop}\label{prop:rcradialstretchings}
Let $F\colon \R^{n\times n}\to \R$ be rank-one convex. For any radial stretching $\phi$, 
$$F(\Id)\leq \fint_{\mb B^n} F(\D \phi) \d x, \qquad
F(\overline \Id)\leq \fint_{\mb B^n} F(\D \bar\phi) \d x.$$
\end{prop}

It is well-known, see for instance \cite{Astala2015,Baernstein1997a,Iwaniec2002}, that the Burkholder integrand is \textit{quasiaffine} on a special class of radial stretchings. Although this fact has not been explicitly proved when $n>2$, the proof is identical.

\begin{prop}\label{prop:Bpradialstretchings}
Let $\phi$ be a radial stretching as in \eqref{eq:defphi} such that
\begin{equation}
\label{eq:nonexpanding}
|\dot \rho(r)|\leq \frac{\rho(r)}{r}.
\end{equation}
The integrand $B_p$ is quasiaffine along such radial stretchings: 
$$B_p(\Id)= \fint_{\mb B^n} B(\D \phi) \d x \text{ if } p\leq n,\qquad
B_p(\overline \Id)= \fint_{\mb B^n} B(\D \bar\phi) \d x \text{ if } p\geq n.
$$
\end{prop}

\begin{proof}
We just deal with the case $\frac n 2 \leq p\leq n$, as the other one is identical. 
Employing \eqref{eq:derivativesphi}, we see that \eqref{eq:nonexpanding} ensures that $|\D\phi(x)|=\frac{\rho(r)}{r}$, and so we calculate
$$B_p(\D\phi) r^{n-1}=
\left[ \bigg(\frac n p -1\bigg)\frac{\rho(r)^n}{r^n} + \frac{\dot \rho(r) \rho(r)^{n-1}}{r^{n-1}}
\right] \frac{\rho(r)^{p-n}}{r^{p-n}}r^{n-1}
= \frac 1 p \frac{\d}{\d r}\left(r^{n-p} \rho(r)^p\right);$$
hence, integrating in spherical coordinates, 
$$\fint_{\mb B^n} B_p(\D \phi) \d x = \frac n p \int_0^1 \frac{\d}{\d r}\left( r^{n-p} \rho(r)^{p}\right) \d r = \frac{n}{p} = B_p(\tp{Id}),$$
where we also used the boundary conditions $\rho(0)=0$ and $\rho(1)=1$.
\end{proof}

\begin{remark}
Condition \eqref{eq:defphi} is known as a \textit{non-expanding} condition: it ensures that $\phi$ does not increase the conformal modulus of annuli centred at zero. The arguments in \cite{GuerraKochLindberg2020c} show that radial stretchings satisfying \eqref{eq:nonexpanding} minimize the $n$-harmonic energy among all maps with the same Jacobian.
\end{remark}

\begin{proof}[Proof of Theorem \ref{thm:extremality}]
Let $\phi$ be a radial stretching which satisfies \eqref{eq:nonexpanding}. As before we just deal with the case $\frac n 2 \leq p\leq n$, as the other one is identical. 
By Propositions \ref{prop:rcradialstretchings} and \ref{prop:Bpradialstretchings},
$$
\fint_{\mb B^n} B_p(\D \phi) \d x= B_p(\Id)=F(\Id)\leq 
\fint_{\mb B^n} F(\D \phi) \d x  .$$
We now take, for some $\a\in [-1,1]$,
$$\rho(r)\equiv\begin{cases}
\frac{r}{2^{\a-1}} & \text{if } r\leq \frac 1 2,\\
r^\a &\text{if } r\geq \frac 1 2.
\end{cases}$$  In particular, $\rho$ satisfies \eqref{eq:nonexpanding} whenever $|\a|\leq 1$. Since $F(\Id)=B_p(\Id)=1$ by assumption and both $F$ and $B_p$ are positively $p$-homogeneous,
$$\int_{\frac 1 2\mb B^n} F(\D \phi)\d x = \int_{\frac 1 2 \mb B^n} B_p(\D \phi) \d x.$$
Writing $\mb A(r_0,r_1)\equiv \{x\in \R^n:r_0<|x|<r_1\}$, it follows that
$$\int_{\mb A(\frac 1 2, 1)} B_p(\D \phi) \d x 
\leq \int_{\mb A(\frac 1 2, 1)} F(\D \phi) \d x .$$
Using \eqref{eq:derivativesphi} and the fact that $|\a|\leq 1$, we have a.e.\ in $\mb A(\frac 1 2, 1)$ the identities
$$|\D \phi(x)|^n= r^{n(\a-1)}, 
\qquad \det \D \phi(x)=\a r^{n(\a-1)}.$$
Note that $B_p$ satisfies \eqref{eq:symmetry}, i.e.\ there is some function $b_p\colon \R^2\to \R$ such that $B_p(A)=b_p(|A|^n,\det A).$
Since $F$ satisfies \eqref{eq:symmetry} as well, we obtain
$$b_p(1,\a)\int_{\frac 1 2}^1  r^{n(\a-1)+n-1}\d r \leq f(1,\a)\int_{\frac 1 2}^1  r^{n(\a-1)+n-1}\d r 
\quad \implies \quad b_p(1,\a)\leq f(1,\alpha). $$
Thus, varying $\a\in [-1,1]$, we have that $ B_p\leq F$ in the segment $[\Id, \overline{\Id}]$.  Since both $F$ and $B_p$ are positively $p$-homogeneous and satisfy \eqref{eq:symmetry}, their values are determined by the values in the segment $[\Id, \overline \Id]$, and so it follows that $B_p\leq F$ at all points.
\end{proof}

\subsection{Polyconvexity}

In this subsection we investigate the polyconvexity properties of the Burkholder integrand. The next result establishes claims \ref{it:necnonneg} and \ref{it:necp>n} in Remark \ref{rmk:optimality}. 

\begin{prop}\label{prop:nonpc}
If $p>2$, the integrand $B_p\colon \R^{2\times 2}\to \R$ is polyconvex nowhere.

If $p<2$ then $B_p^+\colon \R^{2\times 2}\to \R$ is not polyconvex at points where it does not vanish. 
\end{prop}

\begin{proof}
The first part is clear and follows from \cite[Corollary 5.9]{Dacorogna2007}.
The second part follows from the same result, which shows that $B_p^+$ is convex at all points where it is polyconvex. However, $B_p^+$ is not convex at any point where it does not vanish. To see this, we first note that it is not convex at $\Id$, since
$$\frac 1 2 B_p^+(\tp{diag}(2,0))+\frac12 B_p^+(\tp{diag}(0,2))
= 2^p \Big(\frac 2 p -1\Big)< \frac 2 p=B_p^+(\tp{diag}(1,1)).$$ 
Take $A(t)\equiv t \Id + (1-t)\overline \Id$ and note that $t\mapsto  B_p^+(A(t))$ is affine for $t\in [t_0,1]$, where $t_0\equiv 1-\frac{1}{p}$. The convex envelope of $B_p^+$ vanishes at $A(t_0)$, since $B_p^+(A(t_0))=0$, and it is strictly below $B_p^+$ at $A(1)=\Id$. It follows that the convex envelope of $B_p^+$ must be strictly below $B_p^+$ at all points in the segment $(A(t_0),\Id]$. By isotropy and homogeneity, the conclusion follows.
\end{proof}

\begin{remark}
The second statement in Proposition \ref{prop:nonpc} also holds when $n=3$: for $p<3$, $B_p^+\colon \R^{3\times 3}\to \R$ is polyconvex only at points where it vanishes. Indeed, $B_p^+=0$ in $\tp{CO}^-(3)$ and hence, if $F\geq 0$ denotes the polyconvex envelope of $B_p^+$, we must also have $F=0$ in $\tp{CO}^-(3)$. However, by the results of Yan \cite{Yan1997} any polyconvex integrand vanishing on $\tp{CO}^-(3)$ with growth less than $3$ must vanish everywhere, and hence $F=0$ at all points.
\end{remark}

Recall that the case $p=n$ is trivial, since $B_n^+=\det^+$.
Thus Proposition \ref{prop:nonpc} only leaves unanswered the question of whether $B_p^+$ is polyconvex for some $p>n$. The next two results, which correspond to Propositions \ref{prop:Bpn=2} and \ref{prop:Bpn=3}, show that the answer is positive for $n=2$ and negative for $n=3$.

\begin{prop}[$n=2$]\label{prop:pcn=2}
The integrand $B_p^+\colon \R^{2\times2}\to [0,+\infty)$ is polyconvex for $p>2$.
\end{prop}

\begin{proof}
We roughly follow the strategy in \cite{Iwaniec1996}. The crucial point is to perform an appropriate change of variables, which in our case is
$$\begin{cases}
t_A\equiv |A^-|\\
d_A\equiv \det A
\end{cases}
\implies |A^+|=\sqrt{t_A^2+d_A},$$
where we used \eqref{eq:conformalcoords}.
We introduce the function $$h(t,d)\equiv\left(\Big(1-\frac 2 p\Big)(t+\sqrt{t^2+d})^2 + d\right)(t+\sqrt{t^2+d})^{p-2}$$
which satisfies $h(t_A,d_A)=B_p(A)$ and is defined on $S\equiv \{(t,d)\in \R^2: t\geq 0, t^2+d\geq 0\}$. 
We now split the proof into three steps.

\textbf{Step 1:} $h\in C^2$ and $\D^2 h$ is positive semi-definite in the interior of $S$. The first part is clear and the second part follows from elementary calculations. We compute
\begin{align*}
& \p_{tt} h = 2 \frac{(p-1)(p-2)}{\sqrt{t^2+d}}(t+\sqrt{t^2+d})^{p-1},\\
& \p_{td} h = \frac{(p-1)(p-2)}{\sqrt{t^2+d}}(t+\sqrt{t^2+d})^{p-2},\\
&\p_{dd} h =\frac 1 2\frac{(p-1)(p-2)}{ \sqrt{t^2+d}} (t+\sqrt{t^2+d})^{p-3}.
\end{align*}
We see that $\p_{tt}h\geq 0$ while $\det \D^2 h =0$, and so the claim follows.

\textbf{Step 2:} define a function 
$h^+\colon [0,+\infty)\times \R \to [0,+\infty)$ by
$$h^+(t,d)\equiv \begin{cases}
h(t,d) & \text{if } (t,d)\in S' \text{ and } h(t,d)\geq 0,\\
0 &\text{otherwise},
\end{cases}
\quad \text{where } S'\equiv \{(t,d)\in S:h(t,d)\geq 0\}.$$
Note that $S'$ is properly contained in $S$: explicitly,
$$S'= \Big\{(t,d):t\geq 0, \frac{p(p-2)}{(p-1)^2} t^2+d\geq 0 \Big\},\qquad \frac{p(p-2)}{(p-1)^2}\leq 1.$$
This identity seems rather difficult to verify directly from the definition of $h$, but it can be deduced as follows. Since $p>2$, we see from \eqref{eq:defBp} and \eqref{eq:conformalcoords} that
$$B_p(A)=\frac{2}{p}\left((p-1)|A^+|-|A^-|\right) |A|^{p-1}.$$
Thus $B_p(A)\geq 0$ if and only if $(p-1)|A^+|\geq |A^-|$. The expression for $S'$ now follows by recalling the definitions of $h$ and $(t_A,d_A)$.

In this step we show that $h^+$ is convex. The argument is somewhat similar to the one in \cite[Lemma 3.1]{Sverak1990}.
Consider a segment $[a,b]\subset [0,+\infty)\times \R$: we want to show that $h^+$ is convex along $[a,b]$. Consider the set $I=\{x\in [a,b]: x\not \in S'\}$. Clearly $I$ is relatively open in $[a,b]$, $h^+\geq 0$ on $[a,b]\exc I$ and $h^+=0$ on $\p I$. Note that $[a,b]\exc I$ is either a segment or the union of two disjoint segments, and by the previous step $h^+$ is convex in each component of $[a,b]\exc I$. It follows that $h^+$ is convex on $[a,b]$.

\textbf{Step 3:} conclusion. 
Fix $A,B\in \R^{2\times 2}$. By Step 2 there are numbers $\tau_B, \delta_B\in \R$ such that
\begin{align}
\begin{split}
\label{eq:Bppc}
B_p^+(A)-B_p^+(B) & = h^+(t_A,d_A)-h^+(t_B,d_B)\\
&\geq |B|^{p-2}\left(\tau_B (|A^-|-|B^-|)+\delta_B(\det A-\det B)\right).
\end{split}
\end{align}
Since $n=2$, $A\mapsto |A^-|$ is a convex function and the conclusion follows.
\end{proof}

Although the integrands $A\mapsto |A^\pm|$ are convex when $n=2$, they are not even polyconvex when $n>2$: indeed, these integrands are $\frac n2$-homogeneous, c.f.\ \eqref{eq:normconformalpart}, and vanish exactly on $\textup{CO}^\pm(n)$, and so the claim follows from the results of \cite{Yan1997}. In particular, establishing an inequality similar to \eqref{eq:Bppc} in higher dimensions would not imply that $B_p^+$ is polyconvex. In fact, the analogue of Proposition \ref{prop:pcn=2} is already false when $n=3$.

\begin{prop}[$n=3$]
The integrand $B_p^+\colon \R^{3\times 3}\to [0,+\infty)$ is not polyconvex for $p> 3$.
\end{prop}

\begin{proof}
We will show that $B_p^+$ is not polyconvex at the point $\tp{diag}(1,0,0)$. Suppose, on the contrary, that $B_p^+$ it is polyconvex at that point, which is to say that there are constants $c_i$, $i=1,\dots, 7$, such that 
$F(x,y,z)\geq G(x,y,z),$
where
\begin{align*}
&F(x,y,z)\equiv B_p^+(\tp{diag}(x,y,z))-B_p^+(\tp{diag}(1,0,0)),\\
&G(x,y,z)\equiv c_1(x-1) + c_2 y+c_3 z + c_4(x-1)y+c_5(x-1)z + c_6 y z+c_7 (x-1) y z.
\end{align*}
We also note that $F$ is Lipschitz continuous and, by Proposition \ref{prop:Bpprops}, it is differentiable at any point $(x_1,x_2,x_3)\in \R^3$ for which there is some $i$ such that $|x_i|>\max_{j\neq i} |x_j|$.
Since $F(1,0,0)=G(1,0,0)=0$ and $F \geq G$, we must have
\begin{equation}
\label{eq:c1-c3}
\nabla F(1,0,0)=\nabla G(1,0,0) \iff  (p-3,0,0)=(c_1,c_2,c_3).
\end{equation}
We calculate, for $0<y,z<1$,
$$F(1,y,0)=F(1,0,z)=0=G(1,y,0)=G(1,0,z).$$
Note that our restriction on $y,z$ ensures that $F$ is differentiable at $(1,y,0)$ and $(1,0,z)$ and so, since $F\geq G$, we must have
$$\begin{cases}
\nabla F(1,y,0)=\nabla G(1,y,0)\\
\nabla F(1,0,z)=\nabla G(1,0,z)\\
\end{cases}$$
which, using \eqref{eq:c1-c3} above, yields
$$
\begin{cases}
(p-3,0,y)=(p-3+c_4 y,0,c_6 y)\\
(p-3,z,0)=(p-3+c_5 z,c_6 z,0)\\
\end{cases} \iff
(0,0,1)=(c_4,c_5,c_6).$$
In other words, combining \eqref{eq:c1-c3} with this last equation, we actually  have
$$G(x,y,z)=(p-3)(x-1)+y z + c_7 (x-1)y z$$
for some constant $c_7\in \R$. We now calculate
$$G(1,y,y)=y^2, \qquad F(1,y,y)=y^2 \quad \text{if } |y|\leq 1.$$
So we must have, for $|y|<1$, 
$$p-3+c_7y^2 =\p_x G(1,y,y)=\p_x F(1,y,y)=p-3+(p-2)y^2 \iff c_7 = p-2.$$
Finally we evaluate, for $y\geq 0$,
$$G(0,y,-y)=(p-3)(y^2-1),\qquad F(0,y,-y)=(p-3)\frac{(y^p-1)}{p}.$$
It is easy to see that $G(0,y,-y)>F(0,y,-y)$ for $y<1$ sufficiently close to 1: indeed, note that $G(0,1,-1)=F(0,1,-1)=0$, while
$$\left.\frac{\d G(0,y,-y)} {\d y}\right|_{y=1}=2(p-3)>p-3= \left.\frac{\d F(0,y,-y)}{\d y}\right|_{y=1}$$
since $p>3$. This yields the desired contradiction.
\end{proof}

We conclude this paper by deriving an interesting consequence of Proposition \ref{prop:pcn=2}, following a strategy outlined in \cite{Astala2012,Iwaniec2002}. We consider the polyconvex integrands $B_p^+\colon \R^{2\times 2}\to [0,+\infty)$ and we differentiate this family of integrands in $p$ at $p=2$. To be precise, let us fix $A\in \R^{2\times 2}$ such that $\det A>0$; note that, for such $A$, we have $B_p(A)>0$. Since $B_2=\det$, we have
$$
B_\sharp(A)\equiv 
\lim_{p\searrow 2} \frac{B_p^+(A)-\det A}{p-2}=
\left.\frac{\d B_p(A)}{\d p}\right|_{p=2}=\frac 1 2\left(|A|^2+\det A \log(|A|^2)\right).
$$
Since the determinant is a null-Lagrangian, and since $B_p^+$ is polyconvex for every $p\geq 2$, it follows that $B_\sharp\colon \R^{2\times2}_+\to \R$ is also polyconvex. We now consider the involution $\widehat{\cdot}$ acting on integrands $F\colon \R^{n\times n}_+\to \R$ through
$$\widehat{F}(A)\equiv F(A^{-1})\det A.$$
It is well-known that this operation preserves polyconvexity, quasiconvexity and rank-one convexity, c.f.\ \cite[Theorem 2.6]{Ball1977b}. Since $n=2$, $|A^{-1}|=|A|/\det A$, and so
$$\widehat{B_\sharp}(A)=\frac12 \Big(\frac{|A|^2}{\det A} + \log \frac{|A|^2}{(\det A)^2}\Big)=
\frac12 \Big(\frac{|A|^2}{\det A} + \log \frac{|A|^2}{(\det A)}- \log \det A\Big).$$
In particular, we obtain:

\begin{cor}
The integrand $\widehat{B_\sharp}\colon \R^{2\times 2}_+\to \R$ is polyconvex.
\end{cor}

We note that, in \cite{Voss2021}, the authors consider an integrand $W^-_\tp{magic}$ which equals $2\widehat{B_\sharp}$. In particular, the above corollary is equivalent to \cite[Lemma 5.1]{Voss2021}.

\let\oldthebibliography\thebibliography
\let\endoldthebibliography\endthebibliography
\renewenvironment{thebibliography}[1]{
  \begin{oldthebibliography}{#1}
    \setlength{\itemsep}{0.5pt}
    \setlength{\parskip}{0.5pt}
}
{
  \end{oldthebibliography}
}

{\small
\bibliographystyle{abbrv-andre}
\bibliography{/scratch/aguerra/projects/library.bib}
}

\end{document}